\documentclass[12pt]{amsart}
\usepackage{latexsym}
\usepackage{amsfonts}
\usepackage{hyperref}
\usepackage{xcolor}

\usepackage[top=1.2in, bottom=1in, left=1in, right=1in]{geometry}

\newtheorem{theorem}{Theorem}

\newtheorem{lemma}{Lemma}

\DeclareMathOperator{\Aut}{Aut}

\begin{document}

\author[]{Lucy Hyde}
\address{Department of Mathematics, CUNY Graduate Center, New York,
NY 10016}  \email{lhyde@gradcenter.cuny.edu}

\author[]{Siobhan O'Connor}
\address{Department of Mathematics, CUNY Graduate Center, New York,
NY 10016}  \email{doconnor@gradcenter.cuny.edu}

\author[]{Vladimir Shpilrain}
\address{Department of Mathematics, The City College of New York, New York,
NY 10031} \email{shpilrain@yahoo.com}

\title[Orbit-blocking words and Whitehead's problem]{Orbit-blocking words and the average-case complexity of Whitehead's problem in the free group of rank 2}

\begin{abstract}
Let $F_2$ denote the free group of rank 2. Our main technical result of independent interest is: for any element $u$ of $F_2$, there is $g \in F_2$
such that no cyclically reduced image of $u$ under an automorphism of $F_2$ contains $g$ as a subword. We then address computational complexity of the following version of the Whitehead automorphism problem: given a fixed $u \in F_2$, decide, on an input $v \in F_2$ of length $n$, whether or not $v$ is an automorphic image of $u$. We show that there is an algorithm that solves this problem and has constant (i.e., independent of $n$) average-case complexity.

\end{abstract}

\maketitle

\section{Introduction}

The Whitehead problem (see \cite{Wh} or \cite{Lyndonbook}) for a free group is: given two elements, $u$ and $v$, of a free group $F$, find out whether there is an automorphism of $F$ that takes $u$ to $v$.

In this paper, we address computational complexity of the following version of the Whitehead problem: given a fixed $u \in F$, decide, on an input $v \in F$ of length $n$, whether or not $v$ is an automorphic image of $u$.
We show that in the case where the free group has rank 2, there is an algorithm that solves this problem and has constant average-case complexity.

Our main technical result is of independent interest; it settles Problem (F40) from \cite{problems} in the free group $F_2$ of rank 2 (see also Problem 1 and Problem 2 in \cite{VS}):

\begin{theorem}\label{Orbit}
For any element $w$ of $F_2$, there is $g \in F_2$
such that no cyclically reduced image of $w$ under an automorphism of $F_2$ contains $g$ as a subword. Such words $g$ can be produced explicitly for any given $w$.
\end{theorem}

We call elements $g$ like that {\it orbit-blocking} for $w$. This generalizes the idea of {\it primitivity-blocking} words (see e.g. \cite{VS}), i.e., words that cannot be subwords of any cyclically reduced primitive element of a free group. (A {\it primitive element} is part of a free generating set of $F$.) Examples of primitivity-blocking  words can be easily found based on an observation by Whitehead himself (see \cite{Wh} or \cite{Lyndonbook}) that the Whitehead graph of any cyclically reduced primitive element of length $>2$ has either an isolated edge or a cut vertex, i.e., a vertex that, having been removed from the graph together with all incident
edges, increases the number of connected components of the graph. A short and elementary proof of this result was recently given in \cite{Heusener}.

Our technique in the present paper is quite different and is specific to the free group of rank 2. It is based on a description of primitive elements and
{\it primitive pairs} in $F_2$ from \cite{Cohen}. We give more details and a proof of Theorem \ref{Orbit} in Section \ref{blocking}.

In Section \ref{complexity}, based on Theorem \ref{Orbit}, we establish that
the average-case complexity of the version of the Whitehead problem mentioned in the beginning of the Introduction is constant, i.e., is independent of the length of the input $v$. This generalizes a result of \cite{VS} that applies to the special case where the fixed element $u$ is primitive. The result of \cite{VS} though is valid in any free group of a finite rank, whereas our result is limited to $F_2$. Extending it to an arbitrary $F_r$ would require extending Theorem \ref{Orbit} to $F_r$ with $r>2$. While there is little doubt that Theorem \ref{Orbit} holds for any $F_r$, proving it for $r>2$ would require an altogether different approach. More details are given in Section \ref{blocking}.

\section{Orbit-blocking words}\label{blocking}

Let $F_2$ be a free group of rank 2, with generators $a$ and $b$. A {\it primitive pair} in $F_2$ is a pair of words $(u,v)$ such that for some $\varphi \in \Aut(F_2)$, $\varphi(a) = u$ and $\varphi(b) = v$.
Our proof of Theorem \ref{Orbit} relies on the following result from \cite{Cohen} characterizing primitive pairs:
\begin{theorem} \cite{Cohen} \label{basis}
    Suppose that some conjugate of
    $$u=a^{n_1}b^{m_1}\dots a^{n_p}b^{m_p}$$
    and some conjugate of
    $$v=a^{r_1}b^{s_1}\dots a^{r_q}b^{s_q}$$
    form a basis of $F(a,b)$, where $p\geq 1$, $q\geq 1$, and all of the exponents are non-zero. Then, modulo the possible replacement of $a$ by $a^{-1}$ or $b$ by $b^{-1}$ throughout, there are integers $t>0$ and $\varepsilon=\pm 1$ such that either
    \begin{gather*}
        m_1=m_2=\dots=m_p=\varepsilon s_1=\dots=\varepsilon s_q=1, \\
        \{n_1,\dots,n_p,\varepsilon r_1, \dots, \varepsilon r_q\}=\{t,t+1\}
    \end{gather*}
(the latter being an equality of sets)    or, symmetrically,
    \begin{gather*}
        n_1=n_2=\dots=n_p=\varepsilon r_1=\dots=\varepsilon r_q=1, \\
        \{m_1,\dots,m_p,\varepsilon s_1 ,\dots, \varepsilon s_q\}=\{t,t+1\}.
    \end{gather*}
\end{theorem}

The following lemma makes the above description even more specific:

\begin{lemma}\label{PrimitivePair2}
Every primitive pair in $F_2$ is conjugate to a primitive pair where both entries are cyclically reduced.
\end{lemma}

\begin{proof}
The following proof of Lemma \ref{PrimitivePair2} was suggested by the referee.
By way of contradiction, suppose there is a primitive pair $(g^{-1}ug, ~h^{-1}vh)$
of minimum total length. Then $(u, ~gh^{-1}vhg^{-1})$ is a primitive pair, too, and the total length of the latter pair is not greater than that of the original pair.
If $g \ne h$ and there is a cancellation in $gh^{-1}vhg^{-1}$, then we get a primitive pair with a smaller total length, a contradiction. If $g \ne h$ and there is no cancellation in $gh^{-1}vhg^{-1}$, then $(u, ~gh^{-1}vhg^{-1})$ is a Nielsen reduced pair, and therefore any non-trivial product of $u$ and $gh^{-1}vhg^{-1}$ either equals $u^{\pm 1}$ or has length at least 2. Therefore, the subgroup of $F_2$ generated by $u$ and $gh^{-1}vhg^{-1}$ cannot contain both $a$ and $b$, a contradiction.
\end{proof}

For a given word $v$ we will refer to the greatest absolute value of an exponent that appears on $a$ (or $b$) in $v$ as $m_a(v)$ ($m_b(v)$, respectively), omitting $v$ if it is clear from the context. We will refer to the greatest absolute value of an exponent that appears on $a$ (or $b$) in $v$ considered as a cyclic word as $m^\circ_a(v)$ ($m^\circ_b(v)$, respectively).

%Denote the sum of the absolute values of exponents on $a$ whose absolute value is not $1$ in a non-cyclic word $w$ by $s_a(w)$.
%We will use the following lemma for finding bounds on $s_a$ in our proof of Theorem \ref{Orbit}.

\begin{lemma}\label{OneAtATime}
    Let $u_i$ be words in $F_2$ such that $m_a(u_i) \leq 1$ for all $i \in \{1,2,\dots,n\}$. Then $m_a(u_1\cdots u_n)\leq n$ and $m^\circ_a(u_1\cdots u_n) \leq n+1$.
\end{lemma}

\begin{proof}
By way of contradiction, suppose there is a tuple $(u_1,\ldots, u_n)$ 
of minimum total length that does not satisfy the conclusion of the lemma. Then we can 
    assume that there are no cancellations in the products $u_iu_{i+1}$ since otherwise we can choose our $u_i$s differently to decrease $\Sigma |u_i|$ while preserving their product and conditions of the lemma. (For example, if $u_i$ ends with $a$ and $u_{i+1}$ starts with $a^{-1}$, we can replace $u_i$ by $u_ia^{-1}$ and $u_{i+1}$ by $a u_{i+1}$; this will decrease $\Sigma |u_i|$ by 2.) 
    
    Note that since $m_a(u_1)\leq 1$, we have $m^\circ_a(u_1)\leq 2$. By induction, for $k\geq2$, if $m_a(u_1\cdots u_k)\leq k$ and $m^\circ_a(u_1\cdots u_k) \leq k+1$, then $m_a(u_1\cdots u_k u_{k+1})\leq k+1$ and $m^\circ_a(u_1\cdots u_k u_{k+1}) \leq k+2$ since $u_{k+1}$ cannot start or end with $a^{\pm 2}$ by the conditions of the lemma. Thus, our tuple $(u_1,\ldots, u_n)$ will satisfy the conclusion of the lemma, and this contradiction completes the proof.
\end{proof}

\begin{proof}[Proof of Theorem \ref{Orbit}]

    Let $w$ be our given word and $l$ its length. Let $\varphi \in \Aut(F_2)$ send $(a,b)$ to $(u,v)$. By Lemma \ref{PrimitivePair2}, Theorem \ref{basis}, and the fact that swapping $a$ and $b$ is an automorphism, we may assume $m_a(u)$, $m_a(v) \leq 1$.
    Then by Lemma \ref{OneAtATime}, $m_a^\circ(\varphi(w))\leq l+1$. Thus $a^{l+2}b^{l+2}$  cannot appear as a subword of the cyclic reduction of $\varphi(w)$.
\end{proof}

\section{Average-case complexity of the Whitehead problem in $F_2$}\label{complexity}

The idea of the average-case complexity appeared in \cite{Knuth}, formalized in \cite{Levin}, and was addressed in the context of group theory for the first time in \cite{KMSS2}. Specifically, the authors of \cite{KMSS2} addressed the average-case complexity of the word and subgroup membership problems in some non-amenable groups and showed that this complexity was linear.

The strategy (used in \cite{KMSS2}) is, for a given input, to run two algorithms in parallel. One algorithm, call it {\it honest}, always terminates in finite time and gives a correct result. The other algorithm, a {\it Las Vegas algorithm}, is a fast randomized algorithm that never gives an incorrect result; that is, it either  produces the correct result or informs about the failure to obtain any result. (In contrast, a {\it Monte Carlo algorithm} is a randomized algorithm whose output may be incorrect with some (typically small) probability.)

A Las Vegas algorithm can improve the time complexity of an honest, ``hard-working", algorithm that always gives a correct answer but is slow. Specifically, by running a fast Las Vegas algorithm and a slow honest algorithm in parallel, one often gets another honest algorithm whose average-case complexity is somewhere in between because there is a large enough proportion
of inputs on which a fast Las Vegas algorithm will terminate with
the correct answer to dominate the average-case complexity. This idea was used in \cite{KMSS2} where it was shown, in particular, that if a group $G$ has the word problem solvable in subexponential time and if $G$ has a non-amenable factor group where the word problem is solvable in a complexity class $\mathcal{C}$, then there is an honest algorithm that solves the word problem in $G$ with average-case complexity in $\mathcal{C}$. Similar results were obtained for the subgroup membership problem.

We refer to \cite{KMSS2} or \cite{OSh} for formal definitions of the average-case complexity of algorithms working with group words; we chose not to reproduce them here and appeal to intuitive understanding of the average-case complexity of an algorithm as the expected runtime instead.

The word and subgroup membership problems are not the only group-theoretic problems whose average-case complexity can be significantly lower than the worst-case complexity.
In \cite{VS}, it was shown that the average-case complexity of the problem of detecting a primitive element in a free group has {\it constant} time complexity (with respect to the length of the input) if the input is a cyclically reduced word. The same idea was later used in \cite{Roy} to design an algorithm, with constant average-case complexity, for detecting {\it relatively primitive} elements, i.e., elements that are primitive in a given subgroup of a free group.

Here we address computational complexity of the following version of the Whitehead problem: given a fixed $u \in F$, decide, on an input $v \in F$ of length $n$, whether or not $v$ is an automorphic image of $u$. We show that
the average-case complexity of this version of the Whitehead problem is constant if the input $v$ is a cyclically reduced word.

This version is a special case of the general Whitehead algorithm that decides, given two elements $u, v \in F_r$,  whether or not $u$ can be taken to $v$ by an automorphism of $F_r$. The worst-case complexity of the Whitehead algorithm is unknown in general (cf. \cite[Problem (F25)]{problems}) but is at most quadratic in $\max(|u|, |v|)=|u|$ if $r=2$, see \cite{MS} and \cite{Khan}.

We note, in passing, that the generic-case complexity of the Whitehead algorithm was shown to be linear in any $F_r$ \cite{KSS}.
Here we are going to address the average-case complexity of the standard  Whitehead algorithm run in parallel with a fast algorithm that detects ``orbit-blocking" subwords in the input word.

Denote by $B(u)$ a word that cannot occur as a subword of any cyclically reduced $\varphi(u), ~\varphi \in Aut(F_2)$.
Given any particular $u \in F_2$, one can easily produce an orbit-blocking word $B(u)$ based on the argument in our proof of Theorem \ref{Orbit} in Section \ref{blocking}. Specifically, one can use $B(u)$ of the form $a^s b^t$, where $s$ and $t$ are positive integers, each larger than the length of $u$.

We emphasize that in the version of the Whitehead problem that we consider here $u$ is not part of the input. Therefore, constructing $B(u)$ does not contribute to complexity of a solution; it is considered to be pre-computed.

A fast algorithm  $\mathcal{T}$ to detect if $B(u)$ is a subword of a (cyclically reduced) input word $v$ would be as follows. Let $n$ be the length of $v$.  The algorithm  $\mathcal{T}$ would read the initial segments of $v$ of length $k$, $k=1, 2, \ldots,$ adding one letter at a time, and check if this initial segment has $B(u)$ as a subword. This takes time bounded by $C\cdot k$ for some constant $C$, see \cite{Knuth2}.

Denote the ``usual" Whitehead algorithm (that would find out whether or not $v$ is an automorphic image of $u$) by $\mathcal{W}$.
Now we are going to run the algorithms $\mathcal{T}$ and  $\mathcal{W}$ in parallel; denote the composite algorithm by $\mathcal{A}$. Then we have:

\begin{theorem}\label{average-case}
Suppose possible inputs of the above algorithm $\mathcal{A}$ are cyclically reduced words that are selected uniformly at random from the set of cyclically reduced words of length $n$.  Then the average-case time complexity (a.k.a expected runtime) of the algorithm $\mathcal{A}$, working on a classical Turing machine, is $O(1)$, a constant that does not depend on $n$. If one uses the ``Deque" (double-ended queue) model of computing \cite{deque} instead of a classical Turing machine, then the ``cyclically reduced" condition on the input can be dropped. %but the number $n$ should then be considered part of the input.
\end{theorem}

\begin{proof}

Suppose first that the input word $u$ is cyclically reduced.

\noindent {\bf 1.} First we address complexity of the algorithm $\mathcal{T}$. Here we use
a result of \cite{languages} saying that the number of (freely reduced) words of length $L$ with (any number of) forbidden subwords  grows exponentially slower than the number of all freely reduced words of length $L$.

In our situation, we have at least one $B(u)$ as a forbidden subword.  Therefore, the probability that the initial segment of length $k$ of the word $v$ does not have $B(u)$ as a subword is $O(s^k)$ for some $s, ~0<s<1$. Thus, the average time complexity of the algorithm $\mathcal{T}$ is bounded by

\begin{equation}\label{T}
\sum_{k=1}^n  C\cdot k\cdot s^k,
\end{equation}

\noindent which is bounded by a constant.

\medskip

\noindent {\bf 2.} Now suppose that the input word $v$ of length $n$ does not have any subwords $B(u)$, so that we have to rely on the standard Whitehead algorithm $\mathcal{W}$ for an answer. The probability of this to happen is $O(s^n)$ for some $s, ~0<s<1$, as was mentioned before.

The worst-case time complexity of the Whitehead algorithm is known to be  $O(n^2)$ in the group $F_2$ \cite{Khan}.

Thus, the average-case complexity of the composite algorithm $\mathcal{A}$ is

\begin{equation}\label{A}
\sum_{k=1}^n C\cdot k\cdot s^k + O(n^2) \cdot O(s^n),
\end{equation}

\noindent which is bounded by a constant.

\medskip

\noindent {\bf 3.} Now suppose the input word $u$ is not cyclically reduced. Then we are going to cyclically reduce it.
This cannot be done in constant (or even sublinear) time on a classical Turing machine, so here we are going to use the ``Deque" (double-ended queue) model of computing \cite{deque}. It allows one to move between the first and last letter of a word in constant time. We are going to show that with this facility, one can cyclically reduce any element $v$ of length $n$, in any $F_r$, in constant time (on average) with respect to $n=|v|$. In fact, this was previously shown in \cite{VS}, but we reproduce the proof here to make the exposition complete.

First, recall that the number of freely reduced words of length $n$ in $F_r$ is $2r(2r-1)^{n-1}$.

The following algorithm, that we denote by $\mathcal{B}$, will cyclically reduce $v$ on average in constant time with respect to $n=|v|$.

This algorithm will compare the first letter of $v$, call it $a$, to the last letter, call it $z$. If $z \ne a^{-1}$, the algorithm stops right away. If $z = a^{-1}$, the first and last letters are deleted, and the algorithm now works with this new word.

The probability of $z = a^{-1}$ is $\frac{1}{2r}$ for any freely reduced word whose letters were selected uniformly at random from the set  $\{x_1, \ldots, x_r, x_1^{-1}, \ldots, x_r^{-1}\}$. At the next step of the algorithm, however, the letter immediately following $a$ cannot be equal to $a^{-1}$ if we assume that the input is a freely reduced word, so at the next steps (if any) of the algorithm $\mathcal{B}$ the probability of the last letter being equal to the inverse of the first letter will be $\frac{1}{2r-1}$. Then the expected runtime of the algorithm $\mathcal{B}$ on an input word of length $n$ is:

$$\sum_{k=1}^{\frac{n}{2}} \frac{1}{2r} (\frac{1}{2r-1})^{k-1} \cdot k < \sum_{k=1}^\infty (\frac{1}{2r-1})^k \cdot k.$$

\noindent The infinite sum on the right is known to be equal to $\frac{2r-1}{(2r-2)^2}$; in particular, it is constant with respect to $n$.

\end{proof}

\subsection{Acknowledgement} We are grateful to the referee for suggesting several simplifications of our original proof of Theorem 1.

\end{document}